\documentclass[11pt]{scrartcl}
\usepackage{amsmath,amssymb,amsfonts}
\usepackage{latexsym}
\usepackage{epsfig,overpic}
\usepackage{array}
\usepackage{microtype}
\usepackage[colorlinks=true,linkcolor=blue,citecolor=blue,pdfborder={0 0 0}]{hyperref}
\DeclareGraphicsExtensions{.pdf,.eps,.ps,.eps.gz,.ps.gz,.eps.Y}

\usepackage{tikz}
\usetikzlibrary{shapes,arrows,snakes,calendar,matrix,backgrounds,folding,calc,positioning,patterns,spy}

\usepackage{booktabs}

\usepackage{amsthm}
\newtheorem{lemthm}{Base class}[section]
\newtheorem{theorem}[lemthm]{Theorem}
\newtheorem{lemma}[lemthm]{Lemma}

\newtheorem{todo-remark}[lemthm]{ToDo remark}

\newtheorem{remark}{Remark}[section]

\numberwithin{equation}{section}
\numberwithin{figure}{section}
\numberwithin{table}{section}

\definecolor{mildblue}{HTML}{0000b0}
\definecolor{mildpurple}{HTML}{400090}
\definecolor{mildgreen}{HTML}{007000}
\definecolor{mildred}{HTML}{b00000}
\definecolor{lightblue}{HTML}{57c5d9}
\definecolor{lightred}{HTML}{e35555}
\definecolor{grey}{HTML}{BBBBBB}
\definecolor{lightgrey}{HTML}{EEEEEE}
\definecolor{mygrey}{HTML}{777777}
\definecolor{black}{HTML}{000000}
\definecolor{myblockcolor}{HTML}{002e6e}

\newcommand{\bA}{\mathbf A}

\newcommand{\bC}{\mathbf C}

\newcommand{\bn}{\mathbf n}

\newcommand{\bt}{\mathbf \tau}

\newcommand{\bx}{\mathbf x}

\newcommand{\mT}{\mathcal{T}}

\newcommand{\T}{\mathcal T}

\renewcommand{\div}{\mathop{\rm div}}

\newcommand{\rr}{\mathbb{R}}

\newcommand{\jumpleft}{[\![}
\newcommand{\jumpright}{]\!]}

\newcommand{\averageleft}{\{\!\!\{}
\newcommand{\averageright}{\!\}\!\!\}}

\newcommand{\X}{{x}}

\newcommand{\spacejump}[1]{\jumpleft #1 \jumpright}

\newcommand{\average}[1]{\averageleft \! #1 \averageright}

\newcommand{\enorm}[1]{|\!|\!| #1 |\!|\!|}

\newcommand{\MAT}[1]{\mathbf{#1}}

\newcommand{\BrokenOmega}{\Omega_{1\!,2}}

\usepackage{todonotes}

\newcommand{\bT} {\mathbf{T}}
\newcommand{\bbC} {\mathbf{C}}

\newcommand{\W}{\mathcal W}
\newcommand{\J}{\mathcal J}
\newcommand{\s}{\mathcal S}

\def\b|{{\|\hskip-0.16ex|}}
\def\enorm#1{|\!|\!| #1 |\!|\!|_h}
\def\Enorm#1{|\!|\!| #1 |\!|\!|}

\title{Optimal preconditioners for Nitsche-XFEM discretizations of interface problems}

\author{Christoph Lehrenfeld\thanks{Institut f\"ur Geometrie und Praktische  Mathematik, RWTH-Aachen
    University, D-52056 Aachen, Germany (lehrenfeld@igpm.rwth-aachen.de).}
  \and Arnold Reusken\thanks{Institut f\"ur Geometrie und Praktische  Mathematik, RWTH-Aachen
    University, D-52056 Aachen, Germany (reusken@igpm.rwth-aachen.de).}
}

\begin{document}
{

  \maketitle

  \begin{abstract} In the past decade,  a combination of \emph{unfitted} finite elements (or XFEM) with the Nitsche method  has become a popular discretization method for elliptic  interface problems. This development started with the introduction and analysis of this Nitsche-XFEM technique in the paper [A. Hansbo, P. Hansbo, Comput. Methods Appl. Mech. Engrg. 191 (2002)].  In general, the resulting linear systems  have very large  condition numbers, which  depend not only on the mesh size $h$, but also on how the interface intersects the mesh.  
This paper is concerned with the design and analysis of optimal preconditioners for such linear systems. We propose an additive subspace preconditioner which is optimal in the sense that the resulting condition number is independent of the mesh size $h$ and the interface position. We further show that already the simple diagonal scaling of the stifness matrix results in a  condition number that is bounded by $ch^{-2}$, with a constant $c$ that does not depend on the location of the interface.  Both results are proven for the two-dimensional case. Results of numerical  experiments  in two and three dimensions are presented, which illustrate the quality of the preconditioner.
  \end{abstract}
\ \\[2ex]
  \\[2ex]


  \section{Introduction}
  Let $\Omega \in \Bbb{R^d}$, $d=2,3$, be a polygonal domain that is subdivided in two connected subdomains $\Omega_i$, $i=1,2$. For simplicity we assume that $\Omega_1$ is strictly contained in $\Omega$, i.e., $\partial \Omega_1 \cap \partial \Omega =\emptyset$. The interface between the two subdomains is denoted by $\Gamma= \partial \Omega_1 \cap \partial \Omega_2$. 
  We are interested in interface problems of the following type: 
  \begin{subequations} \label{poisson}
    \begin{align}
      \hspace*{2.0cm}  - \div (\alpha \nabla u) = & \, f 
      & \hspace*{-0.6cm} & \text{in}~~ \Omega_i, 
      & \hspace*{-0.6cm} & i=1,2, \hspace*{2.0cm}
      \label{poissoneq1} \\
      \spacejump{\alpha \nabla u \cdot \bn}_{\Gamma}   = & \, 0 
      & \hspace*{-0.6cm} & \text{on}~~ \Gamma, 
      & \hspace*{-0.6cm} & 
      \label{poissoneq2} \\
      \spacejump{\beta u}_{\Gamma} = & \, 0
      & \hspace*{-0.6cm} & \text{on}~~ \Gamma, 
      & \hspace*{-0.6cm} & 
      \label{poissoneq3} \\
      u = & \, 0  
      & \hspace*{-0.6cm} & \text{on}~~\partial \Omega.
      & \hspace*{-0.6cm} & 
      \label{poissoneq4} 
    \end{align}
  \end{subequations}
  Here $\bn$ is the outward pointing unit normal on $\Gamma= \partial \Omega_1$, $ \spacejump{\cdot}$ the usual jump operator and  $\alpha=\alpha_i>0$, $\beta=\beta_i>0$ in $\Omega_i$ are piecewise constant coefficients. In general one has  $\alpha_1 \neq \alpha_2$. If $\beta_1=\beta_2=1$, this is a standard interface problem that is often considered in the literature \cite{Hansbo02,Chen98,burmanzunino12,zawakrbe13}. For $\beta_1 \neq \beta_2$  this model is very similar to models used for mass transport in two-phase flow problems \cite{Bothe03VOF,Bothe04,Sadhal,Slattery,Ishii}. Without loss of generality  we assume $\beta_i \geq 1$. The interface condition in \eqref{poissoneq3} is then usually called the Henry interface condition. Note that if $\beta_1 \neq \beta_2$, the solution $u$ is discontinuous across the interface. If $\beta_1 = \beta_2$ and $\alpha_1 \neq \alpha_2$ the first (normal) derivative of the solution is discontinuous across $\Gamma$. In the setting of two-phase flo
 ws one is typically interested in moving 
interfaces and instead of \eqref{poisson} one uses a time-dependent mass transport model. In this paper, however, we restrict to the simpler stationary case.

  In the past decade,  a combination of \emph{unfitted} finite elements (or XFEM) with the Nitsche method  has become a popular discretization method for this type of interface problems. This development started with the introduction and analysis of this Nitsche-XFEM technique in the paper \cite{Hansbo02}.  Since then this method has been extended in several directions, e.g., as a fictitious domain approach, for the discretization of interface problems in computational mechanics, for the discretization of Stokes interface problems and for the discretization of mass transport problems with moving interfaces, cf.~ \cite{burmanhansbo12,Hansbo04,Hansbo05a,reuskenlehrenfeld12,reuskenlehrenfeld13,ReuskenHieu,HaZa2014}. Almost all papers on this subject treat applications of the method or present discretization error analyses. Efficient iterative solvers for the discrete problem is a topic that has hardly been addressed so far. In general, solving the resulting discrete problem efficiently 
 is a challenging task due 
to the well-known fact that the conditioning of the stiffness matrix is sensitive to the position of the interface relative to the mesh. If the interface cuts elements in such a way that the ratio of the areas (volumes) on both sides of the interface is very large, the stiffness matrix becomes (very) ill-conditioned.  

Recently, for stabilized versions of the Nitsche-XFEM method condition  number bounds of the form $ch^{-2}$, with a constant $c$ that is independent of how the interface $\Gamma$ intersects the triangulation, have been derived \cite{burmanhansbo12,HaZa2014,zawakrbe13}. In \cite{HaZa2014} an inconsistent stabilization is used to guarantee LBB-stability for the pair of finite element spaces used for the Stokes interface problem. This stabilization also improves the conditioning of the stiffness matrix, leading to a $c h^{-2}$ condition number bound. In \cite{zawakrbe13} a stabilized  variant of the Nitsche-XFEM  for the problem \eqref{poisson} is considered. For this method an $ch^{-2}$ condition number bound is derived.

In this paper we consider the original Nitsche-XFEM method from \cite{Hansbo02} for the discretization of \eqref{poisson}, without any stabilization. In \cite{Hansbo02} for this method optimal discretization error bounds are derived. We prove that after a simple diagonal scaling the condition number is bounded by $ch^{-2}$, with a constant $c$ that is independent of how the interface $\Gamma$ intersects the triangulation. We prove that an \emph{optimal preconditioner}, i.e. the condition number of the preconditioned matrix is independent of $h$ and of how the interface $\Gamma$ intersects the triangulation, can be constructed from approximate subspace corrections. If in the subspace spanned by the continuous piecewise linears one applies a standard multigrid preconditioner and in the subspace spanned by the discontinuous finite element functions that are added close to the interface (the xfem basis functions) one applies a simple Jacobi diagonal scaling, the resulting \emph{additive 
 subspace preconditioner 
is optimal}. The latter is the main result of this paper. The analysis uses the very general theory of subspace correction methods \cite{Xu:92a,Yserentant:93}. Our analysis applies to the two-dimensional case ($d=2$), but we expect that a very similar optimality result holds for $d=3$. This claim is supported by results of numerical experiments that are presented. 

  The results derived in this paper also hold (with minor modifications) if in \eqref{poissoneq2}, \eqref{poissoneq3} one has a nonhomogeneous right-hand side. In such a case one has to modify the right-hand side functional in the variational formulation, but the discrete linear operators that describe   the discretization remain the same.

The outline of this paper is as follows. In section~\ref{sectionXFEM} the Nitsche-XFEM method from \cite{Hansbo02} for the discretization of \eqref{poisson} is described. In section~\ref{sect} we study the direct sum splitting of the XFEM space into three subspaces, namely a subspace of continuous piecewise linears, and two subspaces of xfem functions on both sides of the interface. In Theorem~\ref{thmmain1}, which is the main result of this paper, we prove that this is a uniformly stable splitting. Following standard terminology (as in  \cite{Xu:92a,Yserentant:93}) we introduce an additive subspace preconditioner  in section~\ref{sectprecond}. Based on the   stable splitting property the quality of the preconditioner (i.e., the condition number of the preconditioned matrix) can easily be analyzed. In section~\ref{sectExp} we present results of some numerical experiments, both for $d=2$ and $d =3$. 
  \section{The Nitsche-XFEM discretization}
  \label{sectionXFEM}
  In this section we describe the Nitsche-XFEM discretization, which can be found at several places in the literature \cite{Hansbo02,burmanzunino12}. 
  
  Let $\{\T_h\}_{h>0}$ be a family of shape regular simplicial triangulations of $\Omega$. A triangulation  $\T_h$ consists of simplices $T$, with $h_T:={\rm diam}(T)$ and $h:=\max \{ \, h_T~|~ T \in\T_h \}$.  The triangulation is \emph{unfitted}.
  We introduce some notation for \emph{cut elements}, i.e. elements $T \in \T_h$ with $\Gamma \cap T \neq \emptyset $. The subset of these cut elements is denoted by  $\mT_h^{\Gamma} := \{ T\in \T_h~|~ T\cap \Gamma \neq \emptyset \}$. To simplify the presentation and avoid technical details we assume that for all $T \in \mT_h^{\Gamma}$ the intersection $\Gamma_T := T \cap \Gamma$ does not coincide with a subsimplex of $T$ (a face, edge or vertex of $T$). Hence, we assume that $\Gamma_T$ subdivides $T$ into two subdomains $T_i:=T\cap \Omega_i$ with ${\rm meas}_{d}(T_i) >0$. We further assume that there is at least one vertex of $T$ that is inside domain $\Omega_i,~i=1,2$. In the analysis we assume that $\mT_h^{\Gamma}$ is \emph{quasi-uniform}. 

  
  Let $V_h \subset H_0^1(\Omega)$ 
  be the standard finite element space of continuous piecewise linears corresponding to the
  triangulation $\T_h$ with zero boundary values at $\partial \Omega$. Let $\{ \bx_j~|~j=1, \ldots n\}$, with $n=\dim V_h$, be the set of internal vertices in the triangulation. The index set is denoted by $\mathcal J=\{1,\ldots,n\}$. 
  Let
  $(\phi_j)_{j \in \mathcal J}$ be the nodal basis functions in $V_h$, where $\phi_j$ corresponds to the vertex with index $j$. 
  Let $\mathcal J_\Gamma:=\{\, j\in \mathcal J~|~ ~|\Gamma\cap {\rm supp}(\phi_j)| >0\, \}$ be the index set of those basis functions
  the support of which is intersected by $\Gamma$.  The Heaviside function $H_\Gamma$ has the
  values $H_\Gamma(x)=0$ for $x \in \Omega_1$, $H_\Gamma(x)=1$ for $x \in \Omega_2$. Using this,
  for $j \in \mathcal J_\Gamma$ we
  define an enrichment function $\Phi_j (x):= | H_\Gamma(x)-H_\Gamma(\bx_j)|$.
  We introduce additional, so-called \emph{xfem basis functions}
  $\phi_j^\Gamma:= \phi_j \Phi_j$, $j \in \mathcal J_\Gamma$. Note that   $\phi_j^\Gamma(\bx_k)=0$ for all $j \in \mathcal J_\Gamma, \, k \in \mathcal J$. Furthermore, for $j \in \mathcal J_\Gamma $ and $\bx _j \in \Omega_1$, we have ${\rm supp}(\phi_j^\Gamma) \subset \bar \Omega_2 $ and for $\bx _j \in \Omega_2$, we have ${\rm supp}(\phi_j^\Gamma) \subset \bar \Omega_1$. Related to this, the  index set  $\mathcal J_\Gamma$ is partitioned in $\mathcal J_{\Gamma,2}:=\{ j \in \mathcal J_\Gamma~|~ \bx_j \in \Omega_1\}$ and $\mathcal J_{\Gamma,1}:=\mathcal J_\Gamma \setminus \mathcal J_{\Gamma,2}=\{ j \in \mathcal J_\Gamma~|~ \bx_j \in \Omega_2\}$. Hence, for $j \in \mathcal J_{\Gamma,i}$ the xfem basis function $\phi_j^\Gamma$ has its support in $\bar \Omega_i$, $i=1,2$.  The XFEM space is defined by 
  \begin{equation} \label{tmp:ppp}
    V_h^{\Gamma} := V_h \oplus V_{h,1}^\X \oplus V_{h,2}^\X = V_h \oplus V_{h}^\X \quad \mbox{with } V_{h,i}^\X := {\rm span}\{\, \phi_j^\Gamma~|~j \in \mathcal J_{\Gamma,i}\, \},
  \end{equation}
  and $ V_{h}^\X:=V_{h,1}^\X \oplus V_{h,2}^\X$.
  \begin{remark} \rm The XFEM space $V_h^{\Gamma}$ can also be characterized as follows: $v_h \in V_h^{\Gamma}$ if and only if there exist finite element functions $v_1, v_2 \in V_h$ such that $(v_h)_{|\Omega_i} = (v_i)_{|\Omega_i}$, $i=1,2$. From this characterization one easily derives optimal approximation properties of the XFEM space for functions that are piecewise smooth, cf.~\cite{Hansbo02,reusken07}.   
  \end{remark}
In the literature, e.g., \cite{Hansbo02,burmanzunino12}, discretization with the space $V_h^\Gamma$ is also called an \emph{unfitted finite element method}.\\
  An $L^2$-stability property of the basis $(\phi_j)_{j \in \mathcal J} \cup (\phi_j^\Gamma)_{j \in
    \mathcal J_\Gamma}$ of $V_h^\Gamma$ is given in \cite{reusken07}. 

  For the discretization of the equation \eqref{poisson} in the XFEM space we first introduce some notation for scalar products. 
  The $L^2$ scalar product is denoted by $(u,v)_0:= \int_\Omega u v \, dx$. 
  Furthermore we define
  \[
  (u,v)_{1, \BrokenOmega}:=  (\nabla u, \nabla v)_{L^2(\Omega_1)}+ (\nabla u, \nabla v)_{L^2(\Omega_2)}
  , \quad
  u,v \in H^1(\Omega_{1,2}) := H^1({\Omega_1 \cup \Omega_2}),
  \]
  with the semi-norm denoted by $|\cdot |_{1,\BrokenOmega}=(\cdot,\cdot)_{1, \BrokenOmega}^\frac12 $ and norm $\Vert \cdot \Vert_{1,\BrokenOmega} := (\Vert \cdot \Vert_{0}^2 + \vert \cdot \vert_{1,\BrokenOmega}^2)^{\frac12}$.
  On the interface we introduce the scalar product
  \begin{equation}
    (f,g)_\Gamma:=\int_\Gamma f g \, ds
  \end{equation}
  and the mesh-dependent weighted $L^2$ scalar product 
  \begin{equation}
    (f,g)_{\frac12,h,\Gamma}:=h^{-1} \int_\Gamma f g \, ds.
  \end{equation}
  The Nitsche-XFEM discretization of the interface problem \eqref{poisson} reads as follows: \\
  Find $u_h \in V_h^{\Gamma}$ such that
  \begin{equation} \label{discpoisson}
    \begin{split}
      & (\alpha \beta u_h,v_h)_{1, \BrokenOmega} - ( \average{ \alpha \nabla u_h \cdot \bn}, \spacejump{\beta v_h} )_{\Gamma} -( \average{ \alpha \nabla v_h \cdot \bn}, \spacejump{\beta u_h } )_{\Gamma} \\ &  + ( \lambda  \spacejump{\beta u_h}, \spacejump{\beta v_h} )_{\frac12,h,\Gamma} = (\beta f,v_h)_0   \quad \text{for all}~~ \, v_h \in V_h^\Gamma.
    \end{split}
  \end{equation}
  Here  we used 
the average $\average{ w } := \kappa_1 w_1 + \kappa_2 w_2$ with an element-wise constant $\kappa_i = \frac{|T_i|}{|T|}$. This weighting in the averaging is taken from the original paper \cite{Hansbo02}. The stabilization parameter $\lambda \geq 0$ should be taken sufficiently large, $\lambda > c_{\lambda} \max\{\alpha_i\}_{i=1,2}$, with  a suitable constant $c_{\lambda}$ only depending on the shape regularity of $T\in\mT_h$. 


  Discretization error analysis for this method is available in the literature. In \cite{Hansbo02} optimal order discretization error bounds are derived for the case $\beta_1=\beta_2=1$. The case $\beta_1 \neq \beta_2$ is treated in \cite{ReuskenHieu}. 

  For the development and analysis of preconditioners for the discrete problem, without loss of generality we can restrict to the case $\beta_1=\beta_2=1$.  This is due to the following observation. We note that (also if $\beta_1 \neq \beta_2$) we have $\beta v_h \in V_h^\Gamma$ iff $v_h \in V_h^\Gamma$. Thus, by rescaling the test functions $v_h$ and with $\tilde \alpha:=\alpha \beta^{-1}$ the problem \eqref{discpoisson} can be reformulated as follows:  
  Find $ \tilde u_h= \beta u_h \in V_h^{\Gamma}$ such that
  \begin{equation} \label{discpoissonA}
    \begin{split}
      & ( \tilde \alpha \tilde u_h,v_h)_{1, \BrokenOmega} - ( \average{ \tilde \alpha \nabla \tilde u_h \cdot \bn}, \spacejump{v_h} )_{\Gamma} -( \average{ \tilde \alpha \nabla v_h \cdot \bn}, \spacejump{\tilde u_h } )_{\Gamma} \\ &  + ( \lambda  \spacejump{\tilde u_h}, \spacejump{v_h} )_{\frac12,h,\Gamma} = ( f,v_h)_0   \quad \text{for all}~~ \, v_h \in V_h^\Gamma.
    \end{split}
  \end{equation} 

  The stiffness matrices corresponding to \eqref{discpoisson} and \eqref{discpoissonA} are related by a simple basis transformation. In the remainder of the paper we only consider the preconditioning of the stiffness matrix corresponding to \eqref{discpoissonA}. Via the simple basis transformation the solution to \eqref{discpoissonA} directly gives a solution to \eqref{discpoisson}.
  \begin{remark} \rm In certain situations it may be (e.g., due to implementational aspects) less convenient to transform the discrete problem \eqref{discpoisson} into \eqref{discpoissonA}. If one wants to keep the original formulation, it is easy to provide an (optimal) preconditioner for it, given a preconditioner for the transformed problem \eqref{discpoissonA}. We briefly explain this. Let $(\psi_j)_{1 \leq j \leq m}$ denote the basis for $V_h^\Gamma$, and $\bA$, $\tilde \bA$ the stiffness matrices w.r.t. this basis of the problems \eqref{discpoisson} and \eqref{discpoissonA}, respectively. Let $\bT$ be the matrix representation of the mapping $v_h \to \beta^{-1} v_h$, for $v_h \in V_h^\Gamma$, i.e., the $i$-th row of $\bT$ contains the coefficients $t_{i,k}$ such that $\beta^{-1} \psi_i= \sum_{k=1}^m t_{i,k} \psi_k$. Then the relation $\tilde \bA= \bT \bA \bT^T$ holds. Given a preconditioner $\tilde \bbC$ for $\tilde \bA$, we define $\bbC:= \bT^T \tilde \bbC \bT$ as precondition
 er for $\bA$. Due to the equality of spectra,  $\sigma(\bbC \bA) = \sigma(\tilde \bbC \tilde \bA)$, the quality of $\bC$ as a preconditioner for $\bA$ is the same as the quality of $\tilde \bbC$ as a preconditioner for $\tilde \bA$. 
   \end{remark}
  We introduce a compact notation for the symmetric bilinear form used in \eqref{discpoissonA}. For convenience we write $\alpha$ instead of $\tilde \alpha$, and we assume a global constant value for $\lambda$:
  \begin{equation}  \label{defah}
    a_h(u,v) :=( \alpha u,v)_{1, \BrokenOmega} - ( \average{ \alpha \nabla u \cdot \bn}, \spacejump{v} )_{\Gamma} -( \average{\alpha \nabla v \cdot \bn}, \spacejump{u } )_{\Gamma}   +  \lambda ( \spacejump{ u}, \spacejump{v} )_{\frac12,h,\Gamma}.
  \end{equation}
  This bilinear form is well-defined on $V_h^\Gamma \times V_h^\Gamma$. 
  For the analysis we introduce the bilinear form and corresponding norm defined by
  \begin{equation} \label{norm1}
    \enorm{u}^2= |u|_{1, \BrokenOmega}^2 + \lambda \|\spacejump{u}\|_{\frac12,h,\Gamma}^2, \quad u \in V_h^\Gamma.
  \end{equation}

  In \cite{Hansbo02} it is shown that, for $\lambda$ sufficiently large,  the norm corresponding to the Nitsche bilinear form is uniformly equivalent to $\enorm{\cdot}$:
  \begin{equation} \label{normeq}
    a_h(u,u) \sim \enorm{u}^2 \quad \text{for all}~~u \in V_h^\Gamma.
  \end{equation}
  Here and in the remainder we use the symbol $\sim$ to denote two-sided inequalities with constants that are \emph{independent of h and of how the triangulation is intersected by the interface $\Gamma$}. The constants in these inequalities may depend on $\alpha$ and $\lambda$. We also use $\lesssim$ to denote one-sided estimates that have the same uniformity property. In the remainder we assume that $\lambda >0$ is chosen such that \eqref{normeq} holds.
  \section{Stable subspace splitting} \label{sect}
  We will derive an optimal preconditioner for the bilinear form in \eqref{defah} using the theory of \emph{subspace correction methods}. Two excellent overview papers on this topic are \cite{Xu:92a,Yserentant:93}. The theory of subspace correction methods as described in these overview papers is a very general one, with applications to multigrid and to domain decomposition methods. We apply it for a relatively very simple case with three disjoint spaces.  We use the notation and some main results from \cite{Yserentant:93}. It is convenient to adapt our notation to the one of the abstract setting in  \cite{Yserentant:93}. The three subspaces in \eqref{tmp:ppp} are denoted by $\W_0=V_h$, $\W_i=V_{h,i}^\X$, $i=1,2$. Thus we have the direct sum decomposition
  \begin{equation} \label{splitting}
  \s:=V_h^\Gamma= \W_0 \oplus \W_1 \oplus \W_2.
  \end{equation}
  Below $u=u_0+u_1+u_2 \in \s$ always denotes a decompositon with $u_l \in \W_l$, $l=0,1,2$. For the norm induced by the bilinear form $a_h( \cdot,\cdot)$
  we use the notation
  \[
  \|u\|_h:=a_h(u,u)^\frac12, \quad u \in \s.
  \]
  Recall that this norm is uniformly equivalent to $\enorm{\cdot}$, cf.~\eqref{normeq}.
  In theorem~\ref{thmmain1} below we show that the splitting in \eqref{splitting} is stable w.r.t. the norm $\|\cdot\|_h$. 

  The result in the next theorem is the key point in our analysis. We show that the splitting of $\s$ into $\W_0$ and the subspace spanned by the xfem basis functions $\W_1 \oplus \W_2$ is stable. For this we restrict to the two-dimensional case $d=2$. We use a transformation of certain patches to a reference patch on $[0,1]^2$. We first describe this transformation.\\
  We construct a subdivision of $\mT_h^\Gamma$ into patches $\{\omega_k\}$ as follows, cf.~Figure~\ref{sketchpatch}. We first define a subset $\mathcal E$ of all edges that are intersected by $\Gamma$. Consider an edge $E_1$ which is intersected by $\Gamma$ such that one vertex $V_1$ is in $\Omega_1$ and the other, $V_1^\ast$, is in $\Omega_2$. We define this edge as the first element in $\mathcal E$. Now fix one direction along the interface and going in this direction along $\Gamma$ we get an ordered list of all edges intersected by $\Gamma$. As last edge in this list we include the starting edge $E_1$. As the next edge $E_2 \in \mathcal E$ we take the first one after $E_1$ (in the list) that has no common vertex with  $E_1$. As $E_3 \in \mathcal E$ we take the first one after $E_2$ that has no common vertex with  $E_2$, etc.. To avoid technical details we assume that the final edge $E_{N_{\mathcal E}}$ included in $\mathcal E$ coincides with $E_1$.
  By construction we get a numbering of certain vertices as in the left part of Figure \ref{sketchpatch}: edge $E_j$ has vertices $V_j \in \Omega_1$, $V_j^\ast \in \Omega_2$.
  \begin{figure}[ht]
  \begin{center}
      \input{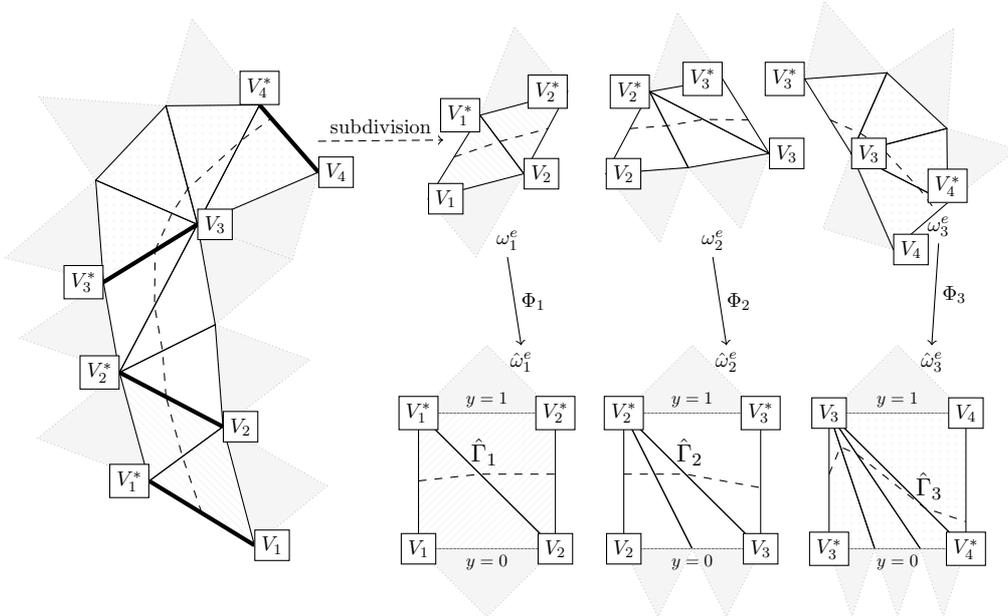}
  \end{center}
  \caption[Sketch of a patch]{Sketch of the partitioning of $\mT_h^\Gamma$ (and neighboring elements) into (extended) patches $\omega_k^e$ and their transformations to a reference configurations.}
  \label{sketchpatch}
\end{figure}
%

   The elements between two edges $E_k, E_{k+1} \in \mathcal{E}$ form the patch $\omega_k$. The patches $\{\omega_k\}_{1 \leq k \leq N_\omega}$, with $N_\omega = N_{\mathcal E} -1$, form a disjoint partitioning of $\T_h^\Gamma$. We define the extended patch $\omega_k^e$ by adding the neighboring elements which are not in $\T_h^\Gamma$, i.e., $\omega_k^e := \omega_k \cup \{ T \in \mT_h \setminus \mT_h^\Gamma~|~\text{$T$ has a common edge with a}~ T' \in \omega_k\}$. The part of the interface $\Gamma$ contained in $\omega_k^e$ is denoted by $\Gamma_k$.  
 The triangulation (and corresponding domain) formed by the union of the extended patches $\omega_k^e$ is denoted by $\T_h^{\Gamma,e}$. Note that every element $T \in \mT_h^{\Gamma,e}$ can appear in at most two patches $\omega_k^e$.
  Further note that the number of elements within each extended patch $\omega_k^e$ is uniformly bounded due to shape regularity of $\mT_h$.
  For each extended patch $\omega_k^e$ there exists a piecewise affine transformation $\Phi_k: \omega_k^e \rightarrow \rr^2$ such that $\Phi_k(\omega_k) = [0,1]^2$. Accordingly we denote a transformed patch by $\hat{\omega}$ and $\hat{\omega}^e$.

  \begin{theorem} \label{thmMain}
    Take $d=2$. The following holds:
    \begin{equation} \label{stable}
      \|u_0\|_h^2 +\|w\|_h^2  \lesssim \|u_0+w\|_h^2 \quad \text{for all}~ u_0 \in \W_0, ~w \in \W_1 \oplus \W_2. 
    \end{equation}
  \end{theorem}
  \begin{proof}
    Due to norm equivalence the result in \eqref{stable} is equivalent to:
\[
    \Enorm{u_0}_h^2 +\Enorm{w}_h^2  \lesssim \Enorm{u_0+w}_h^2 \quad \text{for all}~ u_0 \in \W_0, ~w \in \W_1 \oplus \W_2.
\]
For $w \in \W_1 \oplus \W_2$ we have $w=0$ on $\Omega \setminus \T_h^{\Gamma,e}$,  and $\T_h^{\Gamma,e}$ is partitioned 
into patches $\omega^e_k$. Hence, it suffices to prove
    \begin{equation} \label{res1}
      \Enorm{u_0}_{h,\omega^e_k}^2 +\Enorm{w}_{h,\omega^e_k}^2 \lesssim \Enorm{u_0+ w}_{h,\omega^e_k}^2 \quad \text{for all}~u_0 \in \W_0, ~w \in \W_1 \oplus \W_2.
    \end{equation}
    We use the transformation to the reference patch $\hat \omega^e$  described above. On the reference patch we have transformed spaces $\hat \W_0$ (continuous, piecewise linears) and  $ \hat \W_1 \oplus \hat \W_2$. The functions  in $\hat\W_1$ ($\hat \W_2$) are piecewise linear on the part of the  patch below (above) the interface $\hat \Gamma$, zero on the line segment $y=0$ ($y=1$) and zero on  the part of the  patch above (below) the interface $\hat \Gamma$. 
    The norm $\Enorm{u}_{h,\omega^e_k}$  and the induced norm $ \Enorm{\hat u}_{\hat \omega^e_k}= \big( (\nabla \hat u,\nabla \hat u)_{L^2(\hat \omega_k^e)} +\lambda (\spacejump{ \hat u},\spacejump{\hat u})_{L^2(\hat \Gamma_k)}\big)^\frac12$, with $\hat u = u \circ \Phi_k^{-1}$ on $\hat \omega_k^e$, are uniformly equivalent, because the constants in this norm equivalence are determined only by the condition number of the 
    piecewise affine transformation between $\omega^e_k$ and $\hat \omega_k^e$. Note that neither the spaces $\hat \W_l$ nor the norm  $ \Enorm{\cdot }_{\hat \omega^e_k}$ depend on $h$ (the $h$-dependence is implicit in the piecewise affine transformation).  The reference patches $\hat \omega_k^e$ all have the same geometric structure, cf.~Figure~\ref{sketchpatch}.  These patches have (due to shape regularity of $\T_h$) a uniformly bounded number of vertices on the line segment that connects the vertices $V_i$, $V_{i+1}$ (or $V_i^\ast$, $V_{i+1}^\ast$). In the rest of the proof a generic reference patch and its extension are denoted by $\hat \omega$ and $\hat \omega^e$, respectively. The interface segment that is intersected by $\hat \omega$ is denoted by $\hat \Gamma$. We conclude that for \eqref{res1} to hold it is sufficient to prove
    \begin{equation} \label{res2}
      \Enorm{u_0}_{\hat \omega^e}^2 +\Enorm{w}_{\hat \omega^e}^2 \leq K \Enorm{u_0+ w}_{\hat \omega^e}^2 \quad \text{for all}~u_0 \in \hat W_0, ~w \in \hat W_1 \oplus \hat W_2,
    \end{equation}
    with a constant $K$ that is independent of how the patch $\hat \omega$ is intersected by the interface $\hat \Gamma$.
Note that $(\nabla  u_0,\nabla  w)_{L^2({\hat \omega}^e \setminus \hat \omega)} = (\spacejump{  u_0},\spacejump{ w})_{L^2(\hat \Gamma)} = 0$ for $u_0 \in \hat W_0$ and $w \in \hat W_1 \oplus \hat W_2$. Hence,
\[
  \Enorm{u_0+w}_{\hat \omega^e}^2 = \Enorm{u_0}_{\hat \omega^e}^2 + \Enorm{w}_{\hat \omega^e} ^2 + 2 (\nabla u_0,\nabla w)_{L^2(\hat \omega)}, \quad u_0 \in \hat W_0, ~w \in \hat W_1 \oplus \hat W_2
\]
holds.
 Thus
    it suffices to prove the strengthened Cauchy-Schwarz inequality
    \begin{equation}\label{SCS}
      (\nabla u_0,\nabla w)_{L^2(\hat \omega)} \leq C^\ast \Enorm{u_0}_{\hat \omega^e} \Enorm{w}_{\hat \omega^e} \quad \text{for all}~u_0 \in \hat W_0, ~w \in \hat W_1 \oplus \hat W_2,
    \end{equation}
    with a uniform constant $C^\ast<1$.
    The proof of \eqref{SCS} is divided into three steps, namely a strengthened Cauchy-Schwarz inequality related to the $x$-derivative, a suitable Cauchy-Schwarz inequality related to the $y$-derivative and then combining these estimates. \\
    \underline{Step 1.} The following holds:
    \begin{equation} \label{step1} 
      | (u_x, w_x)_{L^2(\hat \omega)}| \leq c_0 \|u_x\|_{L^2(\hat \omega^e)} \|w_x\|_{L^2(\hat \omega)} \quad \text{for all}~u \in  W_0, ~w \in \hat W_1 \oplus \hat W_2,
    \end{equation}
    with a uniform constant $c_0 <1$. From the  Cauchy-Schwarz inequality  we get
    $ | (u_x, w_x)_{L^2(\hat \omega)}| \leq \|u_x\|_{L^2(\hat \omega)} \|w_x\|_{L^2(\hat \omega)} $. Within the patch $\hat\omega = \{ T_i \}$ the $x$-derivative $u_x$ is piecewise constant and $u_x|_{T_i} = u_x|_{T_{i,N}}$ for the neighboring triangle $T_{i,N} \in \hat{\omega}^e \setminus \hat{\omega}$. This implies $\|u_x\|_{L^2(T_i)} \leq \hat c \|u_x\|_{L^2(T_i \cup T_{i,N})}$, with $\hat c <1$ depending only on shape regularity.  Thus we obtain $ \|u_x\|_{L^2(\hat \omega)} \leq c_0 \|u_x\|_{L^2(\hat \omega^e)}$, with a uniform constant $c_0 <1$, which yields \eqref{step1}.\\
    \underline{Step 2.} The following holds: 
    \begin{equation} \label{step2} 
      \begin{split}
        | (u_y, w_y)_{L^2(\hat \omega)}|  & \leq \min \{ c_1 \|u_x\|_{L^2(\hat \omega)} , \|u_y\|_{L^2(\hat \omega)} \} \|w_y\|_{L^2(\hat \omega)} \\
        & + c_2 \|u_y\|_{L^2(\hat \omega)} \|\spacejump{w}\|_{L^2(\hat \Gamma)}\quad  \text{for all}~u \in  W_0, ~w \in \hat W_1 \oplus \hat W_2,
      \end{split}
    \end{equation}
    with suitable uniform constants $c_1,c_2$.

    Let $\{T_i\}$ be the set of triangles that form $\hat{\omega}$ and let these be ordered such that ${\rm meas}_1 (T_i \cap T_{i+1}) > 0$. We denote the interior edges by $e_i = T_i \cap T_{i+1}$. To show \eqref{step2} we start with partial integration
    \begin{equation}\label{partInt}
      \begin{split}
        \Big| \int_{\hat{\omega}} u_y w_y \, dx \Big| & = \Big| \sum_{T_i} \int_{\partial T_i} n_{T_i,y}\,u_y w \, ds + \int_{\hat \Gamma_{T_i}}  n_{\Gamma,y}\, u_y \spacejump{w} \, ds \Big| \\
        & \leq \sum_{e_i} \Big| \spacejump{u_y}_{e_i} \Big| \Big| \int_{e_i} w \, ds \Big| +\Vert u_y \Vert_{L^2(\hat{\Gamma})} \Vert \spacejump{w} \Vert_{L^2(\hat{\Gamma})} 
      \end{split}
    \end{equation}
    where for the edges of $\partial T_i$ that lie on $\partial \hat{\omega}= \partial [0,1]^2$ we used $w=0$ for $y\in\{0,1\}$ and $n_{T_i,y}=0$ for $x \in \{0,1\}$. To proceed we need  technical estimates to bound $\spacejump{u_y}_{e_i}$ and $\int_{e_i} w \, ds$. For those estimates we exploit propertries of the geometry of $\hat{\omega}$.
    First consider $u \in \hat{W}_0$ along an interior edge $e_i \not\in \partial \hat{\omega}$ and denote the unit tangential vector to $e_i$ by $\bt=(\tau_x,\tau_y)$. For $\tau$ we have $|\tau_y| \geq 1/\sqrt{2} \geq |\tau_x|$. Due to continuity of $u$ along $e_i$ there holds $ \spacejump{\nabla u}_{e_i} \cdot \bt = 0$, which implies 
    \[
     \left|\spacejump{u_y}_{e_i} \right| = \left|\frac{\tau_x}{\tau_y}\right| \left| \spacejump{u_x}_{e_i}\right| \leq \big|{{u_x}_|}_{T_i}\big| + \big|{{u_x}_|}_{T_{i+1}}\big|. 
    \]
  Thus we obtain  
\begin{equation}\label{gradjump}
  \left|\spacejump{u_y}_{e_i} \right| \leq c \, \min \{ \|u_x\|_{L^2(T_i \cup T_{i+1})},\|u_y\|_{L^2(T_i \cup T_{i+1})} \,\}.
\end{equation}

    Next, we consider $w = w_1 +w_2\in \hat{W}_1 \oplus \hat{W}_2$ along the interior edge $e_i$. Let $T_i$ be a triangle adjacent to $e_i$. Without loss of generality we assume that two vertices of $T_i$ are in $\hat \Omega_1$ and we thus have $(w_1)_x = 0$ on $T_i$. 
    We denote the vertices of $e_i$ by $\bx_i=e_i \cap \partial \hat{\omega} \cap \hat{\Omega}_i,~i=1,2$ and the intersection point by $\bx_\Gamma=e_i \cap \hat{\Gamma}$ and define the distances $d_i = \Vert \bx_i - \bx_\Gamma \Vert_2,~i=1,2$. As $w$ is piecewise linear along $e_i$, zero at $\bx_1$, and $(w_1)_x = 0$ on $T_i$, we have $w_1(\bx_\Gamma) =\pm d_1 \tau_y (w_1)_y$. Furthermore:
    $$
    \int_{e_i} w \, ds = \frac12 d_1 w_1(\bx_\Gamma) +  \frac12 d_2 w_2(\bx_\Gamma) = \frac12 (d_1+d_2) w_1(\bx_\Gamma) - \frac12 d_2 \spacejump{w}(\bx_\Gamma).
    $$
    We also have the geometrical information $d_1 \leq d_1+d_2 \leq \sqrt{2}$, $d_1 \leq c |T_i|^\frac12$, $|\hat \Gamma_{T_{i}}| \leq \sqrt{2}$ and $d_2 \leq c |\hat\Gamma_{T_{i}}|^\frac12$. Because $\spacejump{w}$ is linear along $\hat \Gamma_{T_i}$ there also holds $|\hat \Gamma_{T_i}|^\frac12 \left| \spacejump{w}(\bx_\Gamma) \right| \leq c \Vert \spacejump{w} \Vert_{L^2(\hat \Gamma_{T_i})} $. Using these results we get
    \begin{equation}\label{edgeint}
      \Big| \int_{e_i} w \, ds \Big| \leq c \Vert w_y \Vert_{L^2(T_i)} + c \Vert \spacejump{w} \Vert_{L^2(\hat \Gamma_{T_i})}.
    \end{equation}
    From \eqref{gradjump} and \eqref{edgeint}  we obtain
    \begin{equation} \label{boundjumps}
      \sum_{e_i} \Big| \spacejump{u_y}_{e_i} \Big| \Big| \int_{e_i} w \, ds \Big| 
      \leq c \Vert u_y \Vert_{L^2(\hat{\omega})} \Vert \spacejump{w} \Vert_{L^2(\hat \Gamma)} 
      + c \Vert u_x \Vert_{L^2(\hat{\omega})} \Vert w_y \Vert_{L^2(\hat \omega)} .
    \end{equation}
    Combining \eqref{partInt}, \eqref{boundjumps} and the Cauchy-Schwarz inequality $\Big| \int_{\hat{\omega}} u_y w_y \, dx \Big| \leq \Vert u_y \Vert_{L^2(\hat{\omega})} \Vert w_y \Vert_{L^2(\hat{\omega})}$ results in \eqref{step2}. \\
    \underline{Step 3.} The following holds:
    \begin{equation} \label{step3} 
      | (\nabla u,\nabla w)_{L^2(\hat \omega)}| \leq C^\ast \big(\|u_x\|_{L^2(\hat \omega^e)} + \|u_y\|_{L^2(\hat \omega)}\big)^\frac12 \big( \|\nabla w \|_{L^2(\hat \omega)}^2 + \lambda \|\spacejump{w}\|_{L^2(\hat \Gamma)}^2 \big)^\frac12 
    \end{equation}
    for all $u \in  W_0, ~w \in \hat W_1 \oplus \hat W_2$,  with a uniform constant $C^\ast <1$.

    The proof combines the preceding results. We define 
    $\alpha_x = \Vert u_x \Vert_{L^2(\hat{\omega}^e)}$,
    $\beta_x = \Vert w_x \Vert_{L^2(\hat{\omega})}$,
    $\alpha_y = \Vert u_y \Vert_{L^2(\hat{\omega})}$,
    $\beta_y = \Vert w_y \Vert_{L^2(\hat{\omega})}$,
    $\gamma = \Vert \spacejump{w} \Vert_{L^2(\hat\Gamma)}$.
    Then we have with \eqref{step1},  \eqref{step2} and $\theta = \frac{\alpha_x^2}{\alpha_x^2+\alpha_y^2}$, $\alpha = (\alpha_x^2+\alpha_y^2)^{\frac12}$ and $\beta=(\beta_x^2+ \beta_y^2 + \lambda \gamma^2 )^{\frac12}$
    \begin{align*}
      |(\nabla u,\nabla w)_{L^2(\hat \omega)} | &\leq c_0 \alpha_x \beta_x + \min\{c_1 \alpha_x,\alpha_y\} \beta_y + c_2 \alpha_y \gamma \\
      & \leq ( c_0^2 \alpha_x^2 + \min\{c_1^2 \alpha_x^2, \alpha_y^2\} + c_2^2 \alpha_y^2\lambda^{-1} )^{\frac12} (\beta_x^2+ \beta_y^2 + \lambda \gamma^2 )^{\frac12} \\
      & \leq ( c_0^2 \theta + \min \{ c_1^2 \theta, 1- \theta \} + c_2^2 (1-\theta) \lambda^{-1})^{\frac12} \alpha \beta.
    \end{align*}
    One easily sees that $c_0^2 \theta + \min \{ c_1^2 \theta, 1- \theta \} \leq \frac{c_0^2+c_1^2}{1+c_1^2} < 1$. For sufficiently large $\lambda$ ($\lambda > \frac{1+c_1^2}{c_2^2 (1-c_0^2)}$) \eqref{step3} follows for a suitable uniform constant $C^\ast <1$.

    The result \eqref{step3} directly implies \eqref{SCS} and thus the estimate \eqref{stable} holds for $\lambda $ sufficiently large. For different values $\lambda \geq \lambda^\ast$, with $\lambda^\ast$ the critical value for which the norm equivalence \eqref{normeq} holds, the norms $\|\cdot\|_h$ (depending on $\lambda$) are equivalent, with equivalence constants depending only on $\lambda$.  This implies that \eqref{stable} holds for any $\lambda \geq  \lambda^\ast$.
  \end{proof}

In the next lemma we derive the stable splitting property of $\W_1 \oplus \W_2$.
  \begin{lemma} \label{lemma1}
    The following holds:
    \begin{align}
      \|u_l\|_h &\sim |u_l|_{1,\Omega_l} \quad \text{for all} ~u_l \in \W_l \quad \text{and}~l=1,2, \label{lem1a} \\
      \|u_1\|_h^2 +\|u_2\|_h^2 & \lesssim \|u_1+u_2\|_h^2 \quad \text{for all}~ u_1+u_2 \in \W_1 \oplus \W_2. \label{lemm1b}
    \end{align}
  \end{lemma}
  \begin{proof}
    Take $l=1$. We have 
    \begin{equation} \label{pl1} \|u_1\|_h^2\sim \enorm{u_1}^2= |u_1|_{1,\Omega_1}^2 + \lambda \|\spacejump{u_1}\|_{\frac12,h,\Gamma}^2 \sim|u_1|_{1,\Omega_1}^2 + h^{-1} \|u_1\|_{L^2(\Gamma)}^2.
    \end{equation}
    This implies $|u_1|_{1,\Omega_1} \lesssim \|u_1\|_h$. 
    Next we show
    \begin{equation} \label{traceres}  h^{-1} \|u_1\|_{L^2(\Gamma)}^2 \lesssim |u_1|_{1,\Omega_1}^2.
    \end{equation}
    For this, we represent $\Gamma$ locally as the graph of a function $\psi$, with a local coordinate system $(\xi,\eta)$ as in Figure \ref{graph}.
    \begin{figure}[ht]
  \begin{center}
    \begin{tikzpicture}[scale=1.0,rotate=0]
  \begin{scope}[scale=1.0,xshift=0cm,yshift=0cm,rotate=0]
    
    \coordinate (etab) at (1,-0.25);
    \coordinate (etat) at (1,3);

    \coordinate (xil) at (0.75,0);
    \coordinate (xir) at (10,0);

    \coordinate (x2) at (2,0);
    \coordinate (x3) at (4,0);
    \coordinate (xc) at (5,0);
    \coordinate (x4) at (6,0);
    \coordinate (x5) at (8,0);
    \coordinate (x55) at (9,0);

    \coordinate (g1) at (0,1.5);
    \coordinate (g2) at (2,1.8);
    \coordinate (g3) at (4,2.3);
    \coordinate (g4) at (6,2.2);
    \coordinate (g5) at (8,1.9);
    \coordinate (g55) at (9,2.1);
    \coordinate (g6) at (10,2.5);

    \draw [->] (xil) -- (xir) node[below]{$\xi$};
    \draw [->] (etab) -- (etat) node[left]{$\eta$};
    
    \draw [<->] (g55) -- (x55);

    \coordinate (gx55) at ( $ (g55)!0.5!(x55) $ );
    \node[right] at  (gx55) {$\psi(\xi) \leq c h$};

    \draw [-] 
    (g1) to[out=0,in=180] 
    (g2) to[out=0,in=180] 
    (g3) to[out=0,in=180] 
    (g4) to[out=0,in=180] 
    (g5) to[out=0,in=180] 
    (g55) to[out=0,in=180] 
    (g6) node[right]{$\Gamma$};

    \draw [pattern=north east lines,opacity=0.15] 
    (g2) to[out=0,in=180] 
    (g3) to[out=0,in=180] 
    (g4) to[out=0,in=180] 
    (g5) --
    (x5) -- 
    (x2) -- cycle 
    ;


    \coordinate (A1) at (3.1,1);
    \coordinate (A2) at (4.7,1.2);
    \coordinate (A3) at (6.1,1.0);
    \coordinate (A4) at (7.5,0.9);
    \coordinate (A5) at (8,1.05);
    \coordinate (B0) at (2,2.7);
    \coordinate (B1) at (3.25,2.9);
    \coordinate (B2) at (5,2.8);
    \coordinate (B3) at (6.5,2.7);
    \coordinate (B4) at (8,2.9);

    \draw [-,draw opacity=0.55] (A1)
    -- (B0) 
    -- (B1)
    -- (B2)
    -- (B3)
    -- (B4)
    -- (A4)
    -- (A3)
    -- (A2)
    -- (A1)
    -- (B1)
    -- (A2)
    -- (B2)
    -- (A3)
    -- (B3)
    -- (A4);

    \coordinate (A3B2) at ( $ (A3)!0.5!(B2) $ );
    \node[right] at (A3B2) {${\rm supp}(u_1)$};

    \coordinate (A2xc) at ( $ (A2)!0.5!(xc) $ );
    \node[right] at (A2xc) {$u_1=0$};

  \end{scope}
\end{tikzpicture}
  \end{center}
  \vspace*{-0.35cm}
  \caption{Local representation of $\Gamma$ as a graph.}
  \label{graph}
\end{figure}
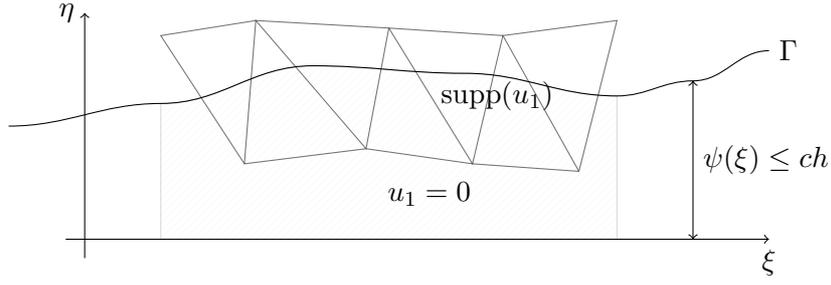
%
    Then we can write
    \begin{align*}
      u_1(\xi,\psi(\xi)) & = \underbrace{u_1(\xi,\psi(0))}_{=0} + \int_{0}^{\psi(\xi)} \frac{\partial u_1}{\partial \eta}(\xi,\eta) \, d\eta, \\
  \end{align*}
and thus
\begin{align*}
 u_1(\xi,\psi(\xi))^2 & = \Big| \int_{0}^{\psi(\xi)} \frac{\partial u_1}{\partial \eta}(\xi,\eta) \, d\eta \Big|^2 \leq \underbrace{|\psi(\xi)|}_{\leq c h} \int_{0}^{\psi(\xi)} (\frac{\partial u_1}{\partial \eta}(\xi,\eta))^2 \, d\eta.
    \end{align*}
    Integration over $\xi$ yields \eqref{traceres}.
    In combination with \eqref{pl1} this yields $\|u_1\|_h^2\lesssim |u_1|_{1,\Omega_1}$, which completes the proof of \eqref{lem1a}. We now consider the result in  \eqref{lemm1b}. Due to $ \| \cdot\|_h \sim \enorm{\cdot}$ is suffices to prove 
    \begin{equation} \label{pl2}
      \enorm{u_1}^2 +\enorm{u_2}^2  \lesssim \enorm{u_1+u_2}^2 \quad \text{for all}~ u_1+u_2 \in \W_1 \oplus \W_2.
    \end{equation}
    The scalar product corresponding to $\enorm{\cdot}$ is denoted by $(\cdot,\cdot)_\ast$, i.e. $(u,v)_\ast= (u,v)_{1, \BrokenOmega} +\lambda(\spacejump{u}, \spacejump{v})_{\frac12,h,\Gamma}$. From $ (u_1,u_2)_{1, \BrokenOmega}=0$ it follows that
    \[
    |(u_1,u_2)_\ast| = | \lambda(\spacejump{u}, \spacejump{v})_{\frac12,h,\Gamma}| \leq \lambda h^{-1} \|u_1\|_{L^2(\Gamma)} \|u_2\|_{L^2(\Gamma)}.
    \]
    Using the results in \eqref{traceres}, \eqref{lem1a} we get, with a suitable constant $c$ and for arbitrary $\delta \in (0,1)$:
    \begin{align*}
      |(u_1,u_2)_\ast| & \leq (1-\delta)\lambda h^{-1} \|u_1 \|_{L^2(\Gamma)} \|u_2\|_{L^2(\Gamma)} + \delta c \lambda |u_1 |_{1,\Omega_1}|u_2|_{1,\Omega_2} \\
      & \leq \max \{1-\delta,\delta c \lambda \} \enorm{u_1} \enorm{u_2}.
    \end{align*}
    By choosing a suitable $\delta$, we obtain the strengthened Cauchy-Schwarz inequality
    \[
    |(u_1,u_2)_\ast| \leq C^\ast\enorm{u_1} \enorm{u_2} \quad \text{for all}~u_1 \in \W_1, \, u_2 \in \W_2,
    \]
    with a constant $ C^\ast < 1$, independent of $h$ and of how the triangulation is intersected by $\Gamma$. This result is equivalent to the one in \eqref{pl2}.
  \end{proof}

  As a direct consequence of the stable splitting properties derived above we obtain the following main result.
  \begin{theorem} \label{thmmain1} Take $d=2$.
    There exists a constant $K$, independent of $h$ and of how the triangulation is intersected by $\Gamma$, such that
    \[  \|u_0\|_h^2 +  \|u_1\|_h^2 + \|u_2\|_h^2 \leq K \|u_0+u_1+u_2\|_h^2 \quad \text{for all}~~u=u_0+u_1+u_2 \in \s.
    \]
  \end{theorem}
  \begin{proof}
    Combine the result in \eqref{stable} with the one in \eqref{lemm1b}.
  \end{proof}
  \section{An optimal preconditioner based on approximate subspace corrections} \label{sectprecond}
  We describe and analyze an additive subspace decomposition preconditioner using the framework given in \cite{Yserentant:93}. 
  For this we first introduce some additional notation. Let $Q_l:\s \to \W_l$, $l=0,1,2$,  be the $L^2$-projection, i.e., for $u \in \s$:
  \[
  (Q_l u, w_l)_0= (u,w_l)_0 \quad \text{for all}~w_l \in \W_l.
  \]
  The bilinear form $a_h(\cdot,\cdot)$ on $\s$ that defines the discretization can be represented by the operator $A:\, \s \to \s$:
  \begin{equation} \label{defA}
  (Au,v)_0 =a_h(u,v) \quad \text{for all}~u,v \in \s.
  \end{equation}
  The discrete problem \eqref{discpoissonA} has the compact representation $A u=f_Q$, where $f_Q$ is the $L^2$-projection of the given data $f \in L^2(\Omega)$ onto the finite element space $\s$.  The Ritz approximations $A_l: \W_l \to \W_l$, $l=0,1,2$, of $A$ are given by
  \[
  (A_l u,v)_0=(Au,v)=a_h(u,v) \quad \text{for all}~ u,v \in \W_l.
  \]
  Note that these are symmetric positive definite operators.
  In the preconditioner we need symmetric positive definite approximations $B_l : \W_l \to \W_l$ of the Ritz operators $A_l$. The spectral equivalence of $B_l$ and $A_l$ is described by the following:
  \begin{equation} \label{spectral}
    \gamma_l ( B_l u,u)_0 \leq (A_l u,u)_0 \leq \rho_l( B_l u,u)_0  \quad \text{for all}~u \in \W_l,
  \end{equation}
  with strictly positive constants $\gamma_l$, $\rho_l$, $l=0,1,2$. The \emph{additive subspace preconditioner} is defined by
  \begin{equation} \label{additive}
    C=\sum_{l=0}^2 B_l^{-1} Q_l.
  \end{equation}
  For the implementation of this preconditioner one has to solve (in parallel) three linear systems. The operator $Q_l$ is not (explicitly) needed in the implementation, since if for a given $z \in \s$ one has to determine $d_l =B_l^{-1} Q_l z$, the solution can be obtained as follows: determine $d_l \in \W_l$ such that
  \[
  (B_l d_l, v)_0=(z,v)_0  \quad \text{for all}~v \in\W_l.
  \]
  The theory presented in \cite{Yserentant:93} can be used to quantify the quality of the preconditioner $C$.
  \begin{theorem} \label{MMain}
    Define $\gamma_{\min}= \min_l \gamma_l$, $\rho_{\max} = \max_l \rho_l$. Let $K$ be the constant of the stable splitting in Theorem~\ref{thmmain1}. The spectrum $\sigma (CA)$ is real and
    \[
    \sigma(CA) \subset \big[ \frac{\gamma_{\min}}{K }, 3 \rho_{\max} \big]
    \]
    holds.
  \end{theorem}
  \begin{proof}
    We recall a main result from \cite{Yserentant:93} (Theorem 8.1). If there are strictly positive constants $K_1, K_2$ such that
    \[
    K_1^{-1} \sum_{l=0}^2 (B_l u_l,u_l) \leq \|u_0+u_1+u_2\|_h^2 \leq K_2 \sum_{l=0}^2 (B_l u_l,u_l)\quad \text{for all}~u_l \in \W_l
    \]
    is satisfied, then $\sigma(CA) \subset [K_1^{-1},K_2]$ holds.
    For the lower inequality we use Theorem~\ref{thmmain1} and \eqref{spectral}, which then results in
    \[
    \|u_0+u_1+u_2\|_h^2 \geq K^{-1} \sum_{l=0}^2\|u_l\|_h^2 = K^{-1} \sum_{l=0}^2 (A_l u_l,u_l)_0 \geq \frac{\gamma_{\min}}{K}  \sum_{l=0}^2 (B_l u_l,u_l)_0.
    \]
    For the upper bound we note
    \[
    \|u_0+u_1+u_2\|_h^2  \leq 3 \sum_{l=0}^2\|u_l\|_h^2 = 3 \sum_{l=0}^2 (A_l u_l,u_l)_0 \leq 3 \rho_{\max}\sum_{l=0}^2 (B_l u_l,u_l)_0. 
    \]
    Now we apply the above-mentioned result with $K_1=K/\gamma_{\min}$ and $K_2=3 \rho_{\max}$.
  \end{proof}
  The result in Theorem~\ref{thmmain1} yields that the constant $K$ is independent of $h$ and of how the triangulation intersects the interface $\Gamma$. It remains to choose appropriate operators $B_l$ such that $\gamma_{\rm min}$ and $\rho_{\max}$ are uniform constants, too.

  We first consider the approximation $B_0$ of the Ritz-projection $A_0$. Note that the finite element functions in $\W_0=V_h$ are continuous across $\Gamma$. This implies that
  \[
  (A_0 u,v)=a_h(u,v)=(\alpha u,v)_{1, \BrokenOmega}= (\alpha \nabla u, \nabla v)_0 \quad \text{for all}~~u,v \in \W_0.
  \]
  Hence, $A_0$ is a standard finite element discretization of a Poisson equation (with a discontinuous diffusion coefficient $\alpha$). As a preconditioner $B_0$ for $A_0$ we can use a standard symmetric multigrid method (which is a multiplicative subspace correction  method). From the literature \cite{HackbuschMGBook,Xu:92a,Yserentant:93}  we know that for this choice of $B_0$ we have spectral inequalities as in \eqref{spectral}, with $\rho_0=1$ and a constant $\gamma_0 >0$ that is independent of $h$ and of how $\Gamma$ intersects the triangulation. 

  It remains to find an appropriate preconditioner $B_l$ of $A_l$, $l=1,2$. For this we propose the simple Jacobi method, i.e., diagonal scaling as a preconditioner for $A_l$, $l=1,2$. We first introduce the operator $B_l$ that represents the Jacobi preconditioner. Recall that $\W_l= {\rm span}\{ \phi_j^\Gamma~|~ j \in \J_{\Gamma,l} \} $.  Elements $u,v \in \W_l$ have unique representations  $u= \sum_{j \in  \J_{\Gamma,l}} \alpha_j \phi_j^\Gamma$, $v=\sum_{j \in  \J_{\Gamma,l}} \beta_j \phi_j^\Gamma$.  In terms of these representations the Jacobi preconditioner is defined by
  \begin{equation} \label{Jacobi}
    (B_l u,v)_0= \sum_{j \in  \J_{\Gamma,l}} \alpha_j \beta_j a_h( \phi_j^\Gamma,\phi_j^\Gamma), \quad u,v \in \W_l, ~ l=1,2.
  \end{equation}
  Note that $ a_h( \phi_j^\Gamma,\phi_j^\Gamma)$ are diagonal entries of the stiffness matrix corresponding to $a_h(\cdot,\cdot)$. The result in the next lemma shows that this diagonal scaling  yields a robust preconditioner for the Ritz operator $A_l$. 
  \begin{lemma} \label{lemspectralJ}
For the Jacobi preconditioner $B_l$ there are strictly positive constants $\gamma_l$, $\rho_l$, independent of $h$ and of how the triangulation is intersected by $\Gamma$ such that 
    \begin{equation} \label{spectralJ}
      \gamma_l ( B_l u,u)_0 \leq (A_l u,u)_0 \leq \rho_l( B_l u,u)_0  \quad \text{for all}~u \in \W_l, ~l=1,2,
    \end{equation}
    holds.
  \end{lemma}
  \begin{proof}
    Take $u= \sum_{j \in  \J_{\Gamma,l}} \alpha_j \phi_j^\Gamma \in  \W_l$. For each $T \in \T_h^\Gamma$ we define  $T_l= T \cap \Omega_l$, and for each $T_l$ we denote by $V(T_l)$ the set vertices of $T$ that are \emph{not} in $\Omega_l$. Note that $V(T_l) \neq \emptyset$ and $V(T_l)$ does not contain all vertices of $T$.  Using \eqref{lem1a} and the construction of the xfem basis functions we get
    \begin{equation} \label{fg} \begin{split}
        ( B_l u,u)_0 & =\sum_{j \in  \J_{\Gamma,l}} \alpha_j^2 a_h( \phi_j^\Gamma,\phi_j^\Gamma) \sim  \sum_{j \in  \J_{\Gamma,l}} \alpha_j^2 |\phi_j^\Gamma|_{1,\Omega_l}^2 \\
        & = \sum_{T \in \T_h^\Gamma} \sum_{j \in V(T_l)} \alpha_j^2 |\phi_j^\Gamma|_{1,T_l}^2 \sim \sum_{T \in \T_h^\Gamma} \sum_{j \in V(T_l)} \alpha_j^2 \Vert \nabla (\phi_j)_{|T} \Vert_2^2 |T_l|.
      \end{split} \end{equation}
    Using \eqref{lem1a} and the fact that $\nabla u$ is a constant vector on each $T_l$ we get, with $\|\cdot\|_2$ the Euclidean vector norm,
    \begin{equation} \label{fg1}
      (A_l u,u)_0= \|u\|_h^2 \sim |u|_{1,\Omega_l}^2 = \sum_{T \in \T_h^\Gamma} \|\nabla u\|_{L^2(T_l)}^2 =  \sum_{T \in \T_h^\Gamma} |T_l| \|(\nabla u)_{|T_l}\|_2^2.
    \end{equation}
    Now note that $(\nabla u)_{|T_l} = \sum_{j  \in V(T_l)} \alpha_j (\nabla \phi_j^\Gamma)_{|T_l}= \sum_{j \in V(T_l)} \alpha_j \nabla (\phi_j)_{|T}$. Because $V(T_l)$ does not contain all vertices of $T$, the vectors in  the set
    $\{ (\nabla \phi_j)_{|T}~|~j\in V(T_l)\}$ are independent and the angles between the vectors depend only on the geometry of the triangulation $\T_h$. This implies that
    \[
    \|(\nabla u)_{|T_l}\|_2^2 \, \sim \sum_{j \in V(T_l)}\alpha_j^2 \|\nabla (\phi_j)_{|T}\|_2^2.
    \]
    Combining this with the results in \eqref{fg} and \eqref{fg1} completes the proof.
  \end{proof}

  \begin{remark}\label{diagprec} \rm
    Instead of an optimal multigrid preconditioner in the subspace $\W_0= V_h$, one can also use a simpler (suboptimal) Jacobi preconditioner, i.e. $B_0$ analogous to \eqref{Jacobi}.  For this choice the spectral constants in \eqref{spectral} are $\gamma_0\sim h^2$ and $\rho_0 \sim 1$. The three subspaces are disjoint and thus if one applies a Jacobi preconditioner in the three  subspaces, the additive subspace preconditioner $C$ in \eqref{additive} coincides with a Jacobi preconditioner for the operator $A$.  From  Theorem \ref{MMain} we can conclude  that $\kappa(CA) \leq c h^{-2}$ holds, with a constant $c$ independent on $h$ and the cut position. Similar uniform $\mathcal{O} (h^{-2})$ condition number bounds have recently been 
derived in the literature, cf. \cite{zawakrbe13} and \cite{burmanzunino12}. 
In these papers, however, for obtaining such a bound an additional 
stabilization term is added to the bilinear form $a_h(\cdot,\cdot)$.
Our analysis shows that although the condition number of the stiffness matrix corresponding to  $a_h(\cdot,\cdot)$ does not have a uniform (w.r.t. the interface cut) bound $ch^{-2}$, a simple diagonal scaling results in a matrix with a spectral condition number that is bounded by $ ch^{-2}$, with a constant $c$ that is independent of how $\Gamma$ is intersected by the triangulation. We note that adding a stabilization as treated \cite{burmanzunino12}  may have a positive effect not only on the condition number, but also on robustness of the discretization w.r.t. large jumps in the diffusion coefficient.  
  \end{remark}
  \begin{remark}\rm
    The assumption $d=2$ is essential only in the proof of Theorem \ref{thmMain}. Concerning a generalization to $d=3$ we note the following. Firstly, it is not obvious how   the subdivision into patches  $\omega_k$ can be  generalized to three space dimensions. Secondly, if $d=2$ then for every element within the reference patch $\hat{\omega}$ we know that the local finite element space on $T \cap \Omega_i$ is one-dimensional which is exploited to characterize the one-sided limit at the interface. 
In three dimensions the local finite element space can be two-dimensional on both parts $T \cap \Omega_i,~i=1,2$ such that it is not obvious how to generalize the proof of Theorem \ref{thmMain}.

Nevertheless, we expect that the result of Theorem~\ref{thmmain1}, hence also the results  on the additive subspace preconditioner,  hold in three space dimensions. This claim is supported by the results of a numerical example with $d=3$, presented in section \ref{example3d}.
  \end{remark}
  \begin{remark} \rm
    For ease of presentation, all dependencies on $\alpha$, especially on the jumps in $\alpha$, have been absorbed in the constants that appear in the estimates.  The results in neither Lemma \ref{lemma1}, Theorem \ref{thmMain} nor Lemma \ref{lemspectralJ} are robust with respect to jumps in $\alpha$. We illustrate the dependence of the quality of the subspace preconditioner on the jumps in $\alpha$ in  a numerical example in section \ref{example2d}. 
  \end{remark}
 \begin{remark} \rm
  Instead of the \emph{additive} preconditioner $C$ in \eqref{additive}, one can also use a \emph{multiplicative} version, cf.~\cite{Yserentant:93}. The optimality of this multiplicative variant, which can be used as a solver or a preconditioner, can easily be derived using the framework given in \cite{Yserentant:93} and the results presented above. 
\end{remark}
  \section{Numerical experiments} \label{sectExp}
In this section results for different subspace correction preconditioners are presented. We consider a discrete interface  problem
of the form: determine $u_h \in V_h^\Gamma$ such that 
\[
   a_h(u_h,v_h) = (f,v_h)_0 \quad \text{for all}~~v_h \in V_h^\Gamma,
\]
 with $a_h(\cdot, \cdot)$ as in\eqref{defah}. We take test problems with $d=2$ and $d=3$. The resulting stiffness matrix, which is the matrix representation of the operator $A$ in \eqref{defA},  is denoted by $\bA$. 
 The matrices corresponding to the Ritz approximations $A_0$ (projection on $V_h$) and 
$A_\X$ (projection on $V_h^\X$) are denoted by $\MAT{A}_0$ and $\MAT{A}_\X$, respectively.
The diagonal matrices ${\rm diag}(\bA)$, ${\rm diag}(\bA_0)$ ${\rm diag}(\bA_x)$ are denoted by $\MAT{D}_{\MAT{A}}$, $\MAT{D}_0$  and $\MAT{D}_x$, respectively. Furthermore,  $\MAT{C}_0$ denotes a preconditioner for $\MAT{A}_0$, for instance a multigrid preconditioner or $ \MAT{C}_0=\MAT{D}_0$. 
  We define the block preconditioners
  \begin{equation} \label{preco}
    \MAT{B}_{\MAT{A}} := \left( \begin{array}{cc} \MAT{A}_0 & \MAT{0} \\ \MAT{0} & \MAT{A}_\X  \end{array} \right), \quad
    \MAT{B}_{\MAT{D}} := \left( \begin{array}{cc} \MAT{A}_0 & \MAT{0} \\ \MAT{0} & \MAT{D}_\X  \end{array} \right), \quad
    \MAT{B}_{\MAT{C}} := \left( \begin{array}{cc} \MAT{C}_0 & \MAT{0} \\ \MAT{0} & \MAT{D}_\X \end{array} \right).
  \end{equation}
The matrix $\MAT{B}_{\MAT{A}}$ corresponds to an additive subspace preconditioner with exact subspace corrections, $\MAT{B}_{\MAT{D}}$ has an exact correction in $V_h$ and an approximate diagonal subspace correction in $V_h^x$, and  $\MAT{B}_{\MAT{C}}$ has approximate subspace corrections in all subspaces.

In the following we study the performance of these preconditioners, in particular their robustness w.r.t. both the variation in the mesh size $h$ and the location of the interface. We also ilustrate the dependence of the condition numbers on $\lambda$ and the diffusivity ratio $\alpha_1/\alpha_2$.  In section \ref{example2d}  we  consider a two-dimensional example with a challenging configuration in the sense that many elements in the mesh have small cuts. This setting allows for a detailed study of the dependencies on $h$, $\alpha_1/\alpha_2$ and $\lambda$. 
In the second example in section \ref{example3d} we consider a three-dimensional analog and apply a multigrid preconditioner $\MAT{C}_0 $ for $ \MAT{A}_0 $. 

  \subsection{Two-dimensional test case} \label{example2d}
  The domain is the unit square $\Omega=[0,1]^2$ with an interface $\Gamma$ which is a square with corners that are rounded off. A sketch  is displayed in Figure~\ref{fignumex1} (left). The rounded square is centered around $\bx_0$, it is denoted as $\Omega_1$. We set the dimensions to $l=0.2$ and $r=0.05$. In the implementation a piecewise linear approximation of $\Gamma$ is used. To investigate conditioning of the system, we consider a situation with many small cuts. 
  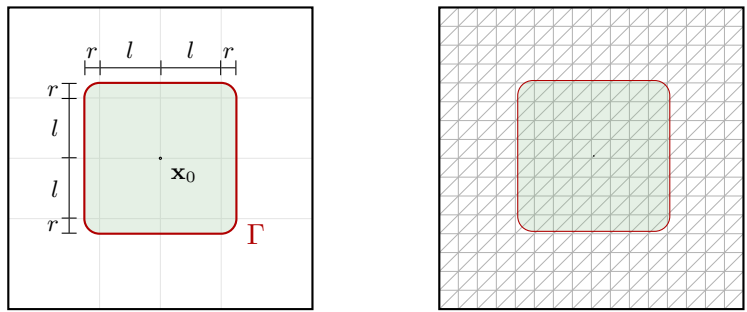
\begin{figure}[ht]
  \begin{center}
    \hspace*{1cm}
    \begin{minipage}{0.3\textwidth}
      \begin{tikzpicture}[scale=4.0]
  \begin{scope}
    \def\interfacecolorstring{mildred}

    \def\cx{0.5}
    \def\cy{0.5}
    \def\d{0.2}
    \def\r{0.05}

    \coordinate (tl) at (0,1);
    \coordinate (tr) at (1,1);
    \coordinate (bl) at (0,0);
    \coordinate (br) at (1,0);

    \coordinate (a1) at (\cx-\d-\r,   \cy+\d);
    \coordinate (a2) at (\cx-\d-\r,\cy+\d+\r);
    \coordinate (a3) at (   \cx-\d,\cy+\d+\r);

    \coordinate (ab) at (\cx      ,\cy+\d+\r);
    \coordinate (abt) at (\cx      ,\cy+\d+2*\r);
    \coordinate (abtt) at (\cx      ,1.0);
    \coordinate (aabt) at (\cx-0.5*\d,\cy+\d+2*\r);
    \coordinate (abbt) at (\cx+0.5*\d,\cy+\d+2*\r);

    \coordinate (b1) at (\cx+\d,   \cy+\d+\r);
    \coordinate (b2) at (\cx+\d+\r,\cy+\d+\r);
    \coordinate (b3) at (\cx+\d+\r,\cy+\d   );

    \coordinate (c1) at (\cx+\d+\r,\cy-\d);
    \coordinate (c2) at (\cx+\d+\r,\cy-\d-\r);
    \coordinate (c3) at (\cx+\d   ,\cy-\d-\r);

    \coordinate (d1) at (\cx-\d   ,\cy-\d-\r);
    \coordinate (d2) at (\cx-\d-\r,\cy-\d-\r);
    \coordinate (d3) at (\cx-\d-\r,\cy-\d   );

    \coordinate (ddal) at (\cx-\d-2*\r,\cy-0.5*\d );
    \coordinate (daal) at (\cx-\d-2*\r,\cy+0.5*\d );
    \coordinate (dal) at (\cx-\d-2*\r,\cy );

    \coordinate (a1l) at (\cx-\d-2*\r,\cy+\d     );
    \coordinate (a2l) at (\cx-\d-2*\r,\cy+\d+  \r);
    \coordinate (a2t) at (\cx-\d-  \r,\cy+\d+2*\r);
    \coordinate (a3t) at (\cx-\d     ,\cy+\d+2*\r);

    \coordinate (a1ll) at (0.0,\cy+\d     );

    \coordinate (a3tt) at (\cx-\d     ,1.0);

    \coordinate (a1la2l) at (\cx-\d-2*\r,\cy+\d+0.5*\r);
    \coordinate (a2ta3t) at (\cx-\d-0.5*\r,\cy+\d+2*\r);

    \coordinate (b1t) at (\cx+\d     ,\cy+\d+2*\r);
    \coordinate (b2t) at (\cx+\d+  \r,\cy+\d+2*\r);
    \coordinate (b2r) at (\cx+\d+2*\r,\cy+\d+  \r);
    \coordinate (b3r) at (\cx+\d+2*\r,\cy+\d     );

    \coordinate (b3rr) at (1.0,\cy+\d     );

    \coordinate (b1tt) at (\cx+\d     ,1.0);

    \coordinate (b1tb2t) at (\cx+\d+0.5*\r,\cy+\d+2*\r);

    \coordinate (c1r) at (\cx+\d+2*\r,\cy-\d     );
    \coordinate (c2r) at (\cx+\d+2*\r,\cy-\d-  \r);
    \coordinate (c2b) at (\cx+\d+  \r,\cy-\d-2*\r);
    \coordinate (c3b) at (\cx+\d     ,\cy-\d-2*\r);

    \coordinate (bcrr) at (1.0,\cy     );

    \coordinate (c1rr) at (1.0,\cy-\d     );
    \coordinate (c3bb) at (\cx+\d     ,0.0);

    \coordinate (d1b) at (\cx-\d     ,\cy-\d-2*\r);
    \coordinate (d2b) at (\cx-\d-  \r,\cy-\d-2*\r);
    \coordinate (d2l) at (\cx-\d-2*\r,\cy-\d-  \r);
    \coordinate (d3l) at (\cx-\d-2*\r,\cy-\d     );

    \coordinate (d2ld3l) at (\cx-\d-2*\r,\cy-\d-0.5*\r);
    \coordinate (d3ll) at (0.0,\cy-\d     );

    \coordinate (dall) at (0.0,\cy );

    \coordinate (d1bb) at (\cx-\d     ,0);

    \coordinate (cdbb) at (\cx ,0.0);

    \draw[black!10] (a3tt) to (d1bb);
    \draw[black!10] (b1tt) to (c3bb);

    \draw[black!10] (abtt) to (cdbb);

    \draw[black!10] (d3ll) to (c1rr);
    \draw[black!10] (a1ll) to (b3rr);

    \draw[black!10] (dall) to (bcrr);

    \draw [thick] (tl) -- (tr) -- (br) -- (bl) -- cycle;

    \draw[thick,\interfacecolorstring,fill=mildgreen!40,opacity=0.25, draw opacity=1] (a1) to[out=90,in=180] (a3)
    to (b1)
    to[out=0,in=90] (b3)
    to (c1)
    to[out=-90,in=0] (c3)
    to (d1)
    to[out=180,in=-90] (d3)
    to (a1);

    \draw[black,|-] (a2t) to (a3t);
    \draw[black,|-] (a3t) to (abt);
    \draw[black,|-|] (abt) to (b1t);
    \draw[black, -|] (b1t) to (b2t);

    \draw[black, |-] (d2l) to (d3l);
    \draw[black,|-]  (d3l) to (dal);
    \draw[black,|-|] (dal) to (a1l);
    \draw[black, -|] (a1l) to (a2l);

    \draw (c2) node[right] {\color{mildred} $\Gamma$};

    \draw (aabt) node[above] {\footnotesize$l$};
    \draw (abbt) node[above] {\footnotesize$l$};

    \draw (a2ta3t) node[above] {\footnotesize$r$};
    \draw (b1tb2t) node[above] {\footnotesize$r$};

    \draw (ddal) node[left] {\footnotesize$l$};
    \draw (daal) node[left] {\footnotesize$l$};

    \draw (a1la2l) node[left] {\footnotesize$r$};
    \draw (d2ld3l) node[left] {\footnotesize$r$};

    \draw (0.5,0.5) circle(0.1pt) node[below right] {\footnotesize $\bx_0$};
  \end{scope}
\end{tikzpicture}

    \end{minipage}
    \hspace*{1cm}
    \begin{minipage}{0.4\textwidth}
      \newcommand\myscalestr{4.0}
\begin{tikzpicture}
  [scale=1.0,
  spy using outlines={black, magnification=4, size=66,
    connect spies, transform shape}
  ]
  \begin{scope}
    \def\interfacecolorstring{mildred}

    \def\cx{0.5078125*\myscalestr} 
    \def\cy{0.5078125*\myscalestr}

    \def\d{0.2*\myscalestr}
    \def\r{0.05*\myscalestr}

    \coordinate (tl) at (0,1*\myscalestr);
    \coordinate (tr) at (1*\myscalestr,1*\myscalestr);
    \coordinate (bl) at (0,0);
    \coordinate (br) at (1*\myscalestr,0);

    \coordinate (a1) at (\cx-\d-\r,   \cy+\d);
    \coordinate (a2) at (\cx-\d-\r,\cy+\d+\r);
    \coordinate (a3) at (   \cx-\d,\cy+\d+\r);

    \coordinate (ab) at (\cx      ,\cy+\d+\r);
    \coordinate (abt) at (\cx      ,\cy+\d+2*\r);
    \coordinate (abtt) at (\cx      ,1.0);
    \coordinate (aabt) at (\cx-0.5*\d,\cy+\d+2*\r);
    \coordinate (abbt) at (\cx+0.5*\d,\cy+\d+2*\r);

    \coordinate (b1) at (\cx+\d,   \cy+\d+\r);
    \coordinate (b2) at (\cx+\d+\r,\cy+\d+\r);
    \coordinate (b3) at (\cx+\d+\r,\cy+\d   );

    \coordinate (c1) at (\cx+\d+\r,\cy-\d);
    \coordinate (c2) at (\cx+\d+\r,\cy-\d-\r);
    \coordinate (c3) at (\cx+\d   ,\cy-\d-\r);

    \coordinate (d1) at (\cx-\d   ,\cy-\d-\r);
    \coordinate (d2) at (\cx-\d-\r,\cy-\d-\r);
    \coordinate (d3) at (\cx-\d-\r,\cy-\d   );

    \coordinate (ddal) at (\cx-\d-2*\r,\cy-0.5*\d );
    \coordinate (daal) at (\cx-\d-2*\r,\cy+0.5*\d );
    \coordinate (dal) at (\cx-\d-2*\r,\cy );

    \coordinate (a1l) at (\cx-\d-2*\r,\cy+\d     );
    \coordinate (a2l) at (\cx-\d-2*\r,\cy+\d+  \r);
    \coordinate (a2t) at (\cx-\d-  \r,\cy+\d+2*\r);
    \coordinate (a3t) at (\cx-\d     ,\cy+\d+2*\r);

    \coordinate (a1ll) at (0.0,\cy+\d     );

    \coordinate (a3tt) at (\cx-\d     ,1.0);

    \coordinate (a1la2l) at (\cx-\d-2*\r,\cy+\d+0.5*\r);
    \coordinate (a2ta3t) at (\cx-\d-0.5*\r,\cy+\d+2*\r);

    \coordinate (b1t) at (\cx+\d     ,\cy+\d+2*\r);
    \coordinate (b2t) at (\cx+\d+  \r,\cy+\d+2*\r);
    \coordinate (b2r) at (\cx+\d+2*\r,\cy+\d+  \r);
    \coordinate (b3r) at (\cx+\d+2*\r,\cy+\d     );

    \coordinate (b3rr) at (1.0,\cy+\d     );

    \coordinate (b1tt) at (\cx+\d     ,1.0*\myscalestr);

    \coordinate (b1tb2t) at (\cx+\d+0.5*\r,\cy+\d+2*\r);

    \coordinate (c1r) at (\cx+\d+2*\r,\cy-\d     );
    \coordinate (c2r) at (\cx+\d+2*\r,\cy-\d-  \r);
    \coordinate (c2b) at (\cx+\d+  \r,\cy-\d-2*\r);
    \coordinate (c3b) at (\cx+\d     ,\cy-\d-2*\r);

    \coordinate (bcrr) at (1.0*\myscalestr,\cy     );

    \coordinate (c1rr) at (1.0*\myscalestr,\cy-\d     );
    \coordinate (c3bb) at (\cx+\d     ,0.0);

    \coordinate (d1b) at (\cx-\d     ,\cy-\d-2*\r);
    \coordinate (d2b) at (\cx-\d-  \r,\cy-\d-2*\r);
    \coordinate (d2l) at (\cx-\d-2*\r,\cy-\d-  \r);
    \coordinate (d3l) at (\cx-\d-2*\r,\cy-\d     );

    \coordinate (d2ld3l) at (\cx-\d-2*\r,\cy-\d-0.5*\r);
    \coordinate (d3ll) at (0.0,\cy-\d     );

    \coordinate (dall) at (0.0,\cy );

    \coordinate (d1bb) at (\cx-\d     ,0);

    \coordinate (cdbb) at (\cx ,0.0);





    \def\dw{0.0625*\myscalestr}
    \foreach \i in {0,...,16} {
      \draw [very thin,black!30] (\i*\dw,0.0) -- (\i*\dw,1.0*\myscalestr); 
    }

    \foreach \j in {0,...,16} {
      \draw [very thin,black!30] (0.0,\j*\dw) -- (1.0*\myscalestr,\j*\dw); 
    }

    \foreach \i in {0,...,15} {
      \foreach \j in {0,...,15} {
        \draw [very thin,black!30] (\i*\dw,\j*\dw) -- (\i*\dw+\dw,\j*\dw+\dw); 
      }
    }

    \draw [thick] (tl) -- (tr) -- (br) -- (bl) -- cycle;

    \draw[\interfacecolorstring,fill=mildgreen!40,opacity=0.25, draw opacity=0] (a1) to[out=90,in=180] (a3)
    to (b1)
    to[out=0,in=90] (b3)
    to (c1)
    to[out=-90,in=0] (c3)
    to (d1)
    to[out=180,in=-90] (d3)
    to (a1);

    \draw[\interfacecolorstring,draw opacity=1] (a1) to[out=90,in=180] (a3)
    to (b1)
    to[out=0,in=90] (b3)
    to (c1)
    to[out=-90,in=0] (c3)
    to (d1)
    to[out=180,in=-90] (d3)
    to (a1);










    \draw (\cx,\cy) circle(0.025*\myscalestr pt) node[below right] {}; 



  \end{scope}
\end{tikzpicture}

    \end{minipage}
  \end{center}
  \caption{Setup of example in section \ref{example2d} (left) and the uniform mesh on level L2 with an interface that generates many small cuts (right).}
  \label{fignumex1}
\end{figure}
%

  To this end we use a uniform triangulation of $\Omega$ and set $\bx_0 = (0.5,0.5) + \varepsilon (1,1)$ with a ``shift parameter'' $\varepsilon = 2^{-20}$. In this configuration almost all cut elements $T \in \mT_h^\Gamma$ have very small cuts (cf. right sketch in Figure \ref{fignumex1}). A similar test case has been considered in \cite{burmanhansbo12} as ``sliver cut case''. We use four levels of uniform refinement  denoted by L1,..,L4. 

  The diffusion parameters are fixed to $(\alpha_1, \alpha_2) = (1.5,2)$. Note that we consider $\beta_1 = \beta_2 = 1$, but the problem is equivalent to every combination of Henry and diffusion parameters which fulfill $(\alpha_1/\beta_1, \alpha_2/\beta_2)= (1.5,2)$. The Nitsche stabilization parameter is set to $\lambda=4 \bar\alpha$ with $\bar\alpha=\frac12 (\alpha_1+\alpha_2)=1.75$. As a right-hand side source term we choose $f=1$ in $\Omega_1$ and $f=0$ in $\Omega_2$. 

In the tables below we present results for the spectral condition number of the preconditioned matrix. We also include the iteration number of the CG method, applied to the preconditioned system,  needed to reduce the starting residual by a factor of $10^{ 6}$.

  \begin{table}[h]
  \begin{center}
      \begin{tabular}{c@{}cc@{}rc@{}rc@{}rc@{}r}
\toprule
&& \multicolumn{2}{c}{L1} & \multicolumn{2}{c}{L2} & \multicolumn{2}{c}{L3} & \multicolumn{2}{c}{L4} \\
\midrule
$\kappa (\MAT{B}_{\MAT{A}}^{-1} \MAT{A})$ &  (its.)
& $4.98 \times 10^0$ & (13) 
& $4.95 \times 10^0$ & (13) 
& $4.82 \times 10^0$ & (12) 
& $4.82 \times 10^0$ & (11)  \\
$\kappa (\MAT{B}_{\MAT{D}}^{-1} \MAT{A})$ &  (its.)
& $5.12 \times 10^0$ & (13) 
& $5.06 \times 10^0$ & (13) 
& $4.94 \times 10^0$ & (12) 
& $4.94 \times 10^0$ & (11)  \\
$\kappa (\MAT{D}_{\MAT{A}}^{-1} \MAT{A})$ &  (its.)
& $2.78 \times 10^1$ & (22) 
& $1.11 \times 10^2$ & (40) 
& $4.42 \times 10^2$ & (73) 
& $1.77 \times 10^3$ & (127) 
\\
\bottomrule
\end{tabular}


  \end{center}
  \caption[Block-precondition]{Condition number and iteration counts of CG method ($\lambda=4 \bar\alpha$, $\alpha_1/\alpha_2=0.75$).}
  \label{tabblockpre}
\end{table}
%

In Table \ref{tabblockpre} the condition numbers corresponding to the block preconditioners $\MAT{B}_{\MAT{A}}$, $\MAT{B}_{\MAT{D}}$ and $\MAT{D}_{\MAT{A}}$ are displayed for four different levels of refinement. The condition number of $\MAT{A}$ is above $10^7$ and the number of CG iterations without preconditioning is above $2000$ on all four levels. We observe that the condition numbers of $\MAT{B}_{\MAT{A}}$ and $ \MAT{B}_{\MAT{D}}$ are essentially independent on the mesh size $h$. 
From further experiments we observe that the condition number of $\MAT{A}$ severely depends on the shift parameter, the results for the block preconditioners however remain essentially the same.  This is in agreement with the results derived in section~\ref{sectprecond}. Also the Jacobi preconditioner $\MAT{D}_{\MAT{A}}$ behaves as expected. With decreasing mesh size $h$, for  the condition number we observe $\kappa(\MAT{D}_{\MAT{A}}^{-1} \MAT{A}) \sim h^{-2}$. 

  \begin{table}[h]
  \begin{center}
      \begin{tabular}{c@{}cc@{}rc@{}rc@{}rc@{}r}
\toprule
\multicolumn{2}{c}{$\lambda/\bar\alpha$} &\multicolumn{2}{c}{$4 \times 10^0$} & \multicolumn{2}{c}{$4 \times 10^1$} & \multicolumn{2}{c}{$4 \times 10^2$} & \multicolumn{2}{c}{$4 \times 10^3$} \\
\midrule
$\kappa (\MAT{B}_{\MAT{A}}^{-1} \MAT{A})$ &  (its.)
& $4.95 \times 10^0$ & (13) 
& $2.50 \times 10^0$ & (9) 
& $2.29 \times 10^0$ & (7) 
& $2.27 \times 10^0$ & (6)  \\
$\kappa (\MAT{B}_{\MAT{D}}^{-1} \MAT{A})$ &  (its.)
& $5.06 \times 10^0$ & (13) 
& $2.14 \times 10^1$ & (13) 
& $2.07 \times 10^2$ & (14) 
& $2.07 \times 10^3$ & (15)  \\
$\kappa (\MAT{D}_{\MAT{A}}^{-1} \MAT{A})$ &  (its.)
& $1.11 \times 10^2$ & (40) 
& $9.49 \times 10^1$ & (36) 
& $2.07 \times 10^2$ & (38) 
& $2.07 \times 10^3$ & (44) 
\\
\bottomrule
\end{tabular}


  \end{center}
  \caption{Condition number and iteration counts of CG method (level L2, $\alpha_1/\alpha_2=0.75$).}
  \label{tablambda}
\end{table}
%

For these preconditioners, with  a fixed mesh (level L2) the dependence on $\lambda$ is shown in Table~\ref{tablambda}.  The results suggest that the estimate in Theorem \ref{thmMain} is essentially independent on $\lambda$. The condition number $\kappa( \MAT{B}_{\MAT{A}}^{-1} \MAT{A})$ even slightly decreases for increasing $\lambda$. 
The diagonal preconditioning of the xfem block $\MAT{A}_x$, however, results in a linear dependence on $\lambda$. Hence, diagonal preconditioning of $\MAT{A}_x$ is not robust w.r.t. $\lambda$. 
Despite the increasing condition number, the CG iteration counts seem to stay almost constant. A similar behavior can be observed for the Jacobi preconditioner $\MAT{D}_{\MAT{A}}$. 

  \begin{table}[h]
  \begin{center}
      \begin{tabular}{c@{}cc@{}rc@{}rc@{}rc@{}r}
\toprule
\multicolumn{2}{c}{$\alpha_1/\alpha_2$} & \multicolumn{2}{c}{$7.5 \times 10^{-1}$} & \multicolumn{2}{c}{$7.5 \times 10^0$} & \multicolumn{2}{c}{$7.5 \times 10^1$} & \multicolumn{2}{c}{$7.5 \times 10^2$} \\
\midrule
$\kappa (\MAT{B}_{\MAT{A}}^{-1} \MAT{A})$ &  (its.)
& $4.95 \times 10^0$ & (13) 
& $1.13 \times 10^1$ & (20) 
& $5.54 \times 10^1$ & (26) 
& $5.21 \times 10^2$ & (28)  \\
$\kappa (\MAT{B}_{\MAT{D}}^{-1} \MAT{A})$ &  (its.)
& $5.06 \times 10^0$ & (13) 
& $1.29 \times 10^1$ & (20) 
& $9.87 \times 10^1$ & (28) 
& $9.61 \times 10^2$ & (26)  \\
$\kappa (\MAT{D}_{\MAT{A}}^{-1} \MAT{A})$ &  (its.)
& $1.11 \times 10^2$ & (40) 
& $6.33 \times 10^2$ & (45) 
& $5.90 \times 10^3$ & (50) 
& $5.86 \times 10^4$ & (72) 
\\
\bottomrule
\end{tabular}


  \end{center}
  \caption{Condition number and iteration counts of CG method (level L2, $\lambda=4\bar\alpha$).}
  \label{tabalpha}
\end{table}
%

In Table \ref{tabalpha} we illustrate the behavior of the preconditioners for increasing diffusivity ratios. We observe that for all three preconditioners the corresponding condition number has a roughly linear dependence on $\alpha_1/\alpha_2$.  We conclude that the stability estimate in Theorem~\ref{thmMain} is not robust with respect to  variation in $\alpha_1/\alpha_2$. The increase of the CG iteration counts, however,  is only very mild. 

  \subsection{Three-dimensional test case} \label{example3d}
We consider a setup in three dimensions very similar to the one used in section \ref{example2d}. The domain is
 the unit cube $\Omega=[0,1]^3$ with a cube that is rounded off as the dividing interface. The cube, denoted as $\Omega_1$, is centered around $\bx_0 = (0.5,0.5,0.5) + \varepsilon (1,1,1)$ with a small ``shift parameter'' $\varepsilon = 2^{-20}$. The dimensions of the cube are chosen as in section \ref{example2d} ($l=0.2$, $r=0.05$) and a uniform triangulation of $\Omega$ is used. We use seven levels of uniform refinement denoted by L0,..,L6 where the coarsest level (L0) is a $2\!\times\!2\!\times\!2$-grid. 

  The diffusion parameters are fixed to $(\alpha_1, \alpha_2) = (1,3)$. Note that we consider $\beta_1 = \beta_2 = 1$. The Nitsche stabilization parameter is set to $\lambda=5 \bar\alpha$ with $\bar\alpha=\frac12 (\alpha_1+\alpha_2)=2$. As a right-hand side source term we choose $f=1$ in $\Omega_1$ and $f=0$ in $\Omega_2$. 

We investigate the performance of the CG method preconditioned with $\MAT{B}_{\MAT{C}}$, cf.~\eqref{preco}. For the preconditioner $\MAT{C}_0$ of  $\MAT{A}_0$ we use  a standard multigrid method.  In this multigrid preconditioner we apply  one V-cycle with a damped Jacobi (damping factor $0.8$) iteration as pre- and post-smoother. 
In Table \ref{tabmgpre} the iteration counts that were needed to reduce the initial residual by a factor of $10^{6}$ for the levels L2 to L6 are shown. On level L6 we have approximately two million unknowns. 

  \begin{table}[h]
  \begin{center}
      \begin{tabular}{cccccc}
\toprule
& L2 & L3 & L4 & L5 & L6 \\
\midrule
CG iterations
 & 22 
 & 25 
 & 27 
 & 29  
 & 32  \\
\bottomrule
\end{tabular}


  \end{center}
  \caption[Multigrid-precondition]{Iteration counts of multigrid-preconditioned CG method ($\lambda=5 \bar\alpha$, $\alpha_2/\alpha_1=3$).}
  \label{tabmgpre}
\end{table}
%

We observe that the iteration counts stay essentially bounded such that the effort for solving the linear systems is $\mathcal{O}(N)$ with $N$ the number of degrees of freedom, i.e. $\MAT{B}_{\MAT{C}}$ is an optimal preconditioner. The mild increase in iteration numbers further decreases if the Jacobi preconditioner $\MAT{D}_x$ used in  the subspace $V_h^\X$ is replaced by a symmetric Gauss-Seidel preconditioner. For this choice we obtain the numbers 21,23,23,25,27 for the levels L2 to L6. 

\subsection*{Acknowledgement}
The authors gratefully acknowledge funding by the German Science Foundation (DFG) within the Priority Program (SPP) 1506 ``Transport Processes at Fluidic Interfaces''. 
  \bibliographystyle{siam}
  \bibliography{literatur}


\end{document}